\documentclass[10pt,reqno]{amsart}
\usepackage{amssymb, amsfonts, amsbsy, latexsym, epsfig, color}

\usepackage{amsthm}
\usepackage{amsmath}
\usepackage{amssymb}
\usepackage{amssymb, amsfonts, amsbsy, latexsym, epsfig, color}
\usepackage{enumerate}
\usepackage{verbatim}

\usepackage{amssymb, amsfonts, amsbsy, latexsym, epsfig, color, amsmath}
\usepackage{enumerate}
\usepackage{amsthm}

\newtheorem{Thm}{Theorem}[section]

\newtheorem{corollary}[Thm]{Corollary}

\newtheorem{proposition}[Thm]{Proposition}
\newtheorem{definition}[Thm]{Definition}
\newtheorem{remark}[Thm]{Remark}

\newtheorem{theorem}[Thm]{Theorem}
\newtheorem*{naimark}{Naimark's Theorem}

\newcommand{\inner}[2]{\langle #1,#2 \rangle}
\newcommand{\norm}[1]{\left\| #1 \right\|}
\newcommand{\abs}[1]{\left| #1 \right|}
\newcommand{\spann}{\mbox{\rm span}}

\def\ldots{\mathinner{\ldotp\ldotp\ldotp}}
\def\ldots{\mathinner{\cdotp\cdotp\cdotp}}

\def\cH{{\mathcal{H}}}

\begin{document}

\title{Every hilbert space frame has a Naimark complement}
\author[P.G. Casazza, M. Fickus, D. Mixon, J. Peterson and I. Smalyanau
 ]{Peter G. Casazza, Matthew Fickus, Dustin G. Mixon, Jesse Peterson and Ihar Smalyanau}
\address{Casazza/Peterson/Smalyanau: Department of Mathematics, University
of Missouri, Columbia, MO 65211-4100}

\email{Casazzap@missouri.edu; Matthew.Fickus@afit.edu; dmixon@princeton.edu;   jdpq6c@mail.missouri.edu}




\maketitle

\begin{abstract}
Naimark complements for Hilbert space Parseval frames are one of the most fundamental and useful results in the field of frame theory.  We will show actually all Hilbert space frames have Naimark complements which possess all the usual properties for Naimark complements of Parseval frames with one notable exception.  Thus these complements can be used with regard to equiangular frames, the restricted isometry property, fusion frames, etc. Along the way, we will correct a mistake in a recent fusion frame paper where chordal distances for Naimark complements are computed incorrectly.
\end{abstract}

\section{Introduction}

Naimark complements for Hilbert space Parseval frames are one of the most fundamental and useful results in the field (see e.g. \cite{Ch}).  A family of vectors
$\{f_n\}_{n=1}^N$ in an $M$-dimensional Hilbert space $\cH_M$ is called a {\em frame} for $\cH_M$ if there are constants $0<A\le B < \infty$
satisfying
\begin{equation} \label{IN1} 
	A \norm{f}^2 \leq \sum_{n=1}^N \abs{\inner{f}{f_n}}^2 \leq B \norm{f}^2, \mbox{ for all } f\in \cH_M.
\end{equation}
The numbers $A,B$ are called {\em lower} and {\em upper frame bounds} of the frame respectively.  If we only require the upper frame bound, we call $\{f_n\}_{n=1}^N$ a $B-${\em Bessel sequence}.  If $A=B$ we call this a $B-${\em tight frame}, and if $A=B=1$, this is a {\em Parseval frame}.  

\begin{naimark}
A family of vectors $\{f_n\}_{n=1}^N$ is a Parseval frame for an $M$-dimensional Hilbert space $\cH_M$ if and only if there is a Hilbert space $\cH_N \supseteq \cH_M$ with an orthonormal basis $\{e_n\}_{n=1}^N$ so that the orthogonal projection $P: \cH_N \rightarrow \cH_M$ satisfies $Pe_n = f_n$ for all $n=1,2,\ldots,N$.  Moreover $\{(I-P)e_n\}_{n=1}^N$ is a Parseval frame for an $(N-M)$-dimensional Hilbert space.  We call such a frame a {\bf Naimark complement} of $\{f_n\}_{n=1}^N$.
\end{naimark}

It is known that many properties of a given Parseval frame carry over to Naimark complements including frame bounds, equal norms among frame vectors, equiangularity among frame vectors, and the restricted isometry property (RIP) to list but a few.  This makes Naimark's theorem useful for finding or constructing frames with specific properties given an existing one \cite{CCHKP, CFMWZ11, FMT}.  Also problems can often be reduced to special cases by switching to Naimark complements, for example the Paulsen problem \cite{CC}.  Naimark's theorem is one of the most used theorems in frame theory.

In this paper we will show all frames, not just Parseval frames, have a natural Naimark complement, and these too carry many basic properties of the frame to the complement with one notable exception. Specifically, the lower frame bound of the Naimark complement may be quite different from the lower frame bound of the original frame.  However, we may calculate this lower frame bound exactly in terms of the eigenvalues of the frame operator of the original frame. 

Fusion frames, originally called frames of subspaces \cite{CK}, are a natural generalization of frames which have developed rapidly due to their applications to problems in sensor networks and distributed processing \cite{CF, CKL}.  The interested reader may see www.fusionframe.org and www.framerc.org for extensive literature on the subject.  Naimark complements in the fusion frames sense, called {\em Naimark fusion frames}, were introduced in \cite{CCHKP}; our concept of a generalized Naimark complement may be considered is this setting as well.  Every fusion frame has a natural complementary Naimark fusion frame, and many properties of these Naimark fusion frames may also be derived from the original fusion frame. Recently, \cite{CCHKP} presented methods for constructing fusion frames with desired properties including chordal distances between subspaces.  This distance is closely related to maximum resillience to noise and erasures when using fusion frames for signal reconstruction.  An incorrect calculation was made while computing this chordal distance, and as we examine the properties of these generalized Naimark fusion frames, we will correct this miscalculation.

This paper is organized as follows.  In section 2 we provide the required basic definitions in frame theory and develop the generalized Naimark complement.  We then demonstrate its similarity to the usual Naimark complement, and we note a significant difference.  In section 3 we examine specific properties of a frame which carry over to its generalized Naimark complement.  Finally, section 4 adapts the new generalized Naimark complement to the setting of fusion frames.


\section{Construction of General Naimark Complements}
Throughout, when given a frame $\{f_n\}_{n=1}^N$ for an $M-$dimensional Hilbert space $\cH_M$ with frame bounds $A,B$, we will assume these are the {\em optimal} values.  That is $A$ and $B$ are the supremum and infimum respectively of all $A$'s and $B$'s satisfying (\ref{IN1}).  The {\em synthesis operator} is $F:\ell_2(N) \rightarrow \cH_M$ given by $F(e_n) = f_n$, where $\{e_n\}_{n=1}^N$ is the canonical orthonormal basis for $\ell_2(N)$ while the {\em analysis operator} of the frame is $F^*:\cH_M \rightarrow  \ell_2(N)$ given by $F^*(f) = (\inner{f}{f_n})_{n=1}^N$.  The {\it frame operator} is then given by $FF^*:\cH_M \rightarrow \cH_M$.  That is, 
\[
	FF^*(f) = \sum_{n=1}^N \inner{f}{f_n} f_n= \sum_{n=1}^N f_nf_n^* f.
\] 
This is a positive, self-adjoint, invertible operator on $\cH_M$.
  
From the matrix point of view, the synthesis operator $F$ is a matrix where the frame vectors form the columns:
\[
	F =  \begin{bmatrix}
					|		& | 	& \cdots & |\\
					f_1 & f_2 & \cdots & f_N\\
					|		& | 	& \cdots & |
				\end{bmatrix}.
\]
In terms of matrix completion, if $F$ is an $M\times N$ Parseval frame, it can be extended by Naimark's theorem to an $N\times N$ unitary matrix by appending $(N-M)$ rows to $F$ to obtain an $N \times N$ unitary matrix
\[
	\begin{bmatrix}
				F\\
				G
	\end{bmatrix}
	 = \begin{bmatrix}
				|		& | 	& \cdots & |\\
				f_1 & f_2 & \cdots & f_N\\
				|		& | 	& \cdots & |\\
				g_1 & g_2 & \cdots & g_N\\
				|		& | 	& \cdots & |\\
			\end{bmatrix}.
\]
In this case, $G=\{g_n\}_{n=1}^N$ is the Naimark complement of $F=\{f_n\}_{n=1}^N$.  This is a slight abuse of notation; we will often make no distinction between frames $\{f_n\}_{n=1}^N$, $\{g_n\}_{n=1}^N$ and their associated $M\times N$ and $(N-M)\times N$ synthesis matrices $F$, $G$.

Given a $B-$Bessel sequence, we will construct a generalized Naimark complement as follows.   We first complete the $B-$Bessel sequence to a $B$-tight frame, and then we obtain a Parseval frame by scaling these frame vectors.  This resulting Parseval frame has a usual Naimark complement; we obtain the general Naimark complement of the $B-$Bessel sequence by then re-scaling and removing the vectors added during the completion.  We will make this all precise, but first, let us recall the proof of Naimark's theorem.

\begin{proof}[Proof of Naimark's Theorem]
Given a Parseval frame $\{f_n\}_{n=1}^N$ for $\cH_M$, the analysis operator $F^*:\cH_M \rightarrow \ell_2(N)$ is the isometry 
\[ 
	F^*f = (\inner{f}{f_1}, \inner{f}{f_2}, \cdots,\inner{f}{f_N}).
\] 
Letting $P$ be the orthogonal projection of $\ell_2(N)$ onto $F^*(\cH_M)$, for any $F^*f$ we have
\begin{align*}
	\inner{F^*f}{Pe_n}	= \inner{F^*f}{e_n} 		
											= \inner{f}{Fe_n}	
											= \inner{f}{f_n}			
											= \inner{F^*f}{F^*f_n}.
\end{align*}
It follows that $Pe_n = F^*f_n$.  Since $F^*$ is an isometry, we may identify $f_n$ with $F^*f_n$, and this completes the proof.
\end{proof}

Our first step in defining the general Naimark complement is to complete a $B-$Bessel sequence to a $B-$tight frame.

\begin{proposition}\label{canon}
Let $\{f_n\}_{n=1}^N$ be a $B-$Bessel sequence in $\cH_M$ with synthesis operator $F$.  Suppose $FF^*$ has eigenvectors $\{\varphi_m\}_{m=1}^M$ with corresponding eigenvalues $B=\lambda_1=\cdots=\lambda_K>\lambda_{K+1}\geq \cdots \geq\lambda_M$.  Set
\[
	h_m:=(B-\lambda_m)^{\frac{1}{2}}\varphi_m
\]
for $K+1\leq m \leq M$.  Then $\{f_n\}_{n=1}^N \cup \{h_m\}_{m=K+1}^M$ is a $B-$tight frame.
\end{proposition}

\begin{proof}
Let $H$ be the synthesis matrix for $\{h_m\}_{m=K+1}^M$ so that $\begin{bmatrix}F &H \end{bmatrix}$ is the synthesis matrix for $\{f_n\}_{n=1}^N \cup \{h_m\}_{m=K+1}^M$.  The associated frame operator is then
\begin{align*}
	\begin{bmatrix}F &H \end{bmatrix}\begin{bmatrix}F &H \end{bmatrix}^* &=FF^*+HH^*
				= \sum_{m=1}^M \lambda_m \varphi_m\varphi_m^* + \sum_{m=K+1}^{M}(B-\lambda_m)\varphi_m\varphi_m^*
				=B\sum_{m=1}^M \varphi_m\varphi_m^*
				=BI,
\end{align*}
and $\{f_n\}_{n=1}^N \cup \{h_m\}_{m=K+1}^M$ is a $B-$tight frame.
\end{proof}

Notice the frame operator for $\{f_n\}_{n=1}^N \cup \{h_m\}_{m=K+1}^\ell$ where $K+1\leq \ell\leq M$ is given by $FF^*+\sum_{m=1}^\ell h_mh_m^*$.  Due to a classic result involving the interlacing of eigenvalues when adding rank one operators (see for example Theorem 4.3.8 in \cite{HJ}), at fewest $M-K$ vectors must be unioned with $\{f_n\}_{n=1}^N$ to yield a $B-$tight frame.  Note the completion in Proposition \ref{canon} appends a minimal number of vectors; we will call this completion the {\it canonical completion} of $\{f_n\}_{n=1}^N$.  As the canonical completion can be scaled to a Pareval frame which has a usual Naimark complement, this allows us to define general Naimark complements for $B-$Bessel sequences.

\begin{definition} \label{naimarkdef}
Let $\{f_n\}_{n=1}^N$ be a $B-$Bessel sequence in $\cH_M$ and $\{f_n\}_{n=1}^N \cup \{h_m\}_{m=K+1}^M$ the canonical completion to a $B-$tight frame.  Let $\Phi^*: \cH_{M}\rightarrow \ell_2(M+N-K)$ be the analysis operator for the Parseval frame $\{\frac{1}{\sqrt{B}}f_n\}_{n=1}^N \cup \{\frac{1}{\sqrt{B}}h_m\}_{m=K+1}^{M}$.  Then $\Phi^*$ is an isometry, and by the proof of Naimark's theorem, there exists an orthonormal basis  $\{e_n\}_{n=1}^{M+N-K}$ for $\ell_2(M+N-K)$ so that $\sqrt{B}Pe_n = \Phi^*f_n$ for all $n=1,\ldots,N$ where $P$ is the orthogonal projection onto $\Phi^*(\cH_M)$.  Then
\[
	\{g_n\}_{n=1}^N := \{\sqrt{B}(I-P)e_n\}_{n=1}^N
\]
is a {\it general Naimark complement} of $\{f_n\}_{n=1}^N$.
\end{definition}

\begin{remark} \label{rmk}
Formally, a general Naimark complement is embedded in $\ell_2(M+N-K)$.  Notice however 
\[
	\Phi^*f_n+g_n= \sqrt{B}e_n  \mbox{ and } \Phi^*f_n\perp g_n
\]
for all $n=1,\ldots,N$.  Although we discard the vectors $\{h_m\}_{m=K+1}^M$ in the definition of the general Naimark complement, setting $h'_m=\sqrt{B}(I-P)e_{N-K+m}$ for $m=K+1, \ldots, M$, we also have
\[
	\Phi^*h_m+h'_m= \sqrt{B}e_{N-K+m} \mbox{ and } \Phi^*h_m\perp h'_m.
\]
Therefore by identifying $\cH_M$ with $\Phi^*(\cH_M)\subseteq \ell_2(M+N-K)$ and $\cH_{N-K}$ with $[\Phi^*(\cH_M)]^{\perp}\subseteq \ell_2(M+N-K)$, we may consider $f_n,h_m\in \cH_M$, and $g_n,h'_m\in \cH_{N-K}$ where $\cH_M\oplus \cH_{N-K}=\ell_2(M+N-K)$.  These identifications allow us to consider Naimark complements from the standpoint of matrix completion.  To be clear, if the $B-$Bessel sequence $\{f_n\}_{n=1}^N$ has the canonical completion $\{f_n\}_{n=1}^N \cup \{h_m\}_{m=K+1}^M$ and a general Naimark complement $\{g_n\}_{n=1}^N$, this identification allows us to set $F$ as the $M\times N$ synthesis matrix of $\{f_n\}_{n=1}^N$, $H_1$ as the $M\times (M-K)$ synthesis matrix of $\{h_m\}_{m=K+1}^M$, $G$ as the $(N-K) \times N$ synthesis matrix of $\{g_n\}_{n=1}^N$, and $H_2$ as the $(N-K) \times (M-K)$ synthesis matrix of $\{h'_m\}_{m=K+1}^M$ so that
\[
	\frac{1}{\sqrt{B}}\begin{bmatrix} F	&	H_1	\\	G	&	H_2	\end{bmatrix}
\]
is a $(M+N-K)\times (M+N-K)$ unitary matrix.  Many of our proofs will consider Naimark complements from this standpoint of matrix completion.
\end{remark}

We now prove several properties for general Naimark complements which are similar to those for the usual Naimark complements of Parseval frames.  Note if $\{f_n\}_{n=1}^N$ is a Parseval frame, then the general Naimark complement is the same as the usual Naimark complement.  For this reason, we now drop the adjective ``general'' and refer to general Naimark complements of $B-$Bessel sequences simply as Naimark complements.  Also, given a $B-$Bessel sequence $\{f_n\}_{n=1}^N$ for $\cH_M$ with Naimark complement $\{g_n\}_{n=1}^N$ in $\cH_{N-K}$, we will assume the embedding discussed in remark \ref{rmk} so that $\cH_M \oplus \cH_{N-K}=\ell_2(M+N-K)$.

\begin{theorem} \label{prop}
If $\{f_n\}_{n=1}^N$ is a $B-$Bessel sequence in $\cH_M$ whose frame operator has eigenvalues $B=\lambda_1=\cdots=\lambda_K>\lambda_{K+1}\geq \cdots \geq \lambda_M$, a Naimark complement $\{g_n\}_{n=1}^N$ satisfies
the following.
\begin{enumerate}[(a)]
	\item \label{one} $\{f_n\oplus g_n\}_{n=1}^N$ is an orthogonal set with $ \norm{f_n\oplus g_n}^2= B$, $n=1,\ldots,N$.
	\item \label{two} $\spann (\{g_n\}_{n=1}^N)=\cH_M^{\perp}=\cH_{N-K} \subset \cH_{M+N-K}$.
	\item \label{three} We have uniqueness in the sense that if $\{\psi_n\}_{n=1}^N$ is another Naimark complement, then there exists a unitary operator $U$ such that $Ug_n=\psi_n$ for all $n=1,\ldots,N$.
\end{enumerate}
\end{theorem}

\begin{proof}
Result (\ref{one}) is clear since $f_n \oplus g_n = \sqrt{B}e_n$ for $n=1,\ldots N$ where $\{e_n\}_{n=1}^{M+N-K}$ is an orthonormal basis.

To show (\ref{two}), consider $\{f_n\}_{n=1}^N$ in its synthesis matrix form $F$ as well as the synthesis matrix $G$ of $\{g_n\}_{n=1}^N$.  Let $F':=\begin{bmatrix}F & H_1 \end{bmatrix}$ be the synthesis matrix for the canonical completion of $\{f_n\}_{n=1}^N$ to a $B-$tight frame.  Then by the definition of a Naimark complement, there exists $H_2$ so that by setting $G':=\begin{bmatrix}G&H_2\end{bmatrix}$,
\begin{equation}\label{mat}
	\frac{1}{\sqrt{B}}\begin{bmatrix}
 				F& H_1 \\
				G & H_2\\
			\end{bmatrix}
			=\frac{1}{\sqrt{B}}\begin{bmatrix}
				F' \\
				G'
			\end{bmatrix}
\end{equation}
is an $(M+N-K)\times (M+N-K)$ unitary matrix.  Note the columns of $G'$ span $\cH_M^{\perp}=\cH_{N-K}\subset\cH_{M+N-K}$ so that $\mbox{rank}(G')=N-K$.  We also have
\begin{align*}
	I_{M+N-K}	=\frac{1}{B}	
		\begin{bmatrix}
 				F^*		& G^* 	\\
				H^*_1	&H^*_2	
		\end{bmatrix}
		\begin{bmatrix}
 				F	& H_1 \\
				G & H_2	
		\end{bmatrix} 
					=\frac{1}{B}
		\begin{bmatrix}
 				F^*F+G^*G	& 0 								\\
				0 				& H_1^*H_1+H_2^*H_2	
		\end{bmatrix}
\end{align*}
so that the $N\times N$ Gram matrix of the Naimark complement $G^*G=BI_N-F^*F$.  As this is diagonalizable,
\begin{align} 
	G^*G	&= \mbox{diag}(B-\lambda_1,\ldots,B-\lambda_K,B-\lambda_{K+1},\ldots,B-\lambda_M,B,\ldots,B) \nonumber\\
				&= \mbox{diag}(0,\ldots,0,B-\lambda_{K+1},\ldots,B-\lambda_M,B,\ldots,B) \label{bnd}.
\end{align}
This shows $\mbox{rank}(G)=\mbox{rank}(G^*G)=N-K=\mbox{rank}(G')$.  It follows that $G$ and $G'$ have the same column space, and thus $\spann (\{g_n\}_{n=1}^N)=\cH_M^{\perp}=\cH_{N-K}\subset\cH_{M+N-K}$.

To prove (\ref{three}), given another Naimark complement $\{\psi_n\}_{n=1}^N$ with synthesis matrix $\Psi$, we consider unitary matrices as in (\ref{mat}).  That is there exist $H_2$ and $H'_2$ such that
\[
	U_1:=\frac{1}{\sqrt{B}}\begin{bmatrix}
 				F& H_1 \\
				G & H_2\\
			\end{bmatrix}
 \mbox{ and }
 	U_2:=\frac{1}{\sqrt{B}}\begin{bmatrix}
 				F& H_1 \\
				\Psi & H'_2\\
			\end{bmatrix}
\]  
are each unitary.  Thus there exists a unitary operator $V:\cH_M\oplus \cH_{N-K} \rightarrow \cH_M\oplus \cH_{N-K}$ such that $VU_1=U_2$.  However, since $V|_{\cH_M}=I$, we have $V=I\oplus U$ where $U:\cH_{N-K}\rightarrow \cH_{N-K}$ is unitary.  Thus $UG=\Psi$, and the result follows.
\end{proof}

Compared to usual Naimark complements of Parseval frames, Naimark complements of non-Parseval frames have a key difference.  Given a Parseval frame, its Naimark complement is a Parseval frame and therefore has frame bounds $A=B=1$.  For a general frame, the frame bounds of a Naimark complement may be very different compared to those of the original frame; however, we can calculate these bounds based on the eigenvalues of the original frame operator.

\begin{theorem}\label{bds}
If $\{f_n\}_{n=1}^N$ is a frame for $\cH_M$ whose frame operator has eigenvalues $B=\lambda_1=\cdots=\lambda_K>\lambda_{K+1}\geq \cdots \geq \lambda_M$, then a Naimark complement $\{g_n\}_{n=1}^N$ is a frame for $\cH_{N-K}$ with lower frame bound $B-\lambda_{K+1}$ and upper frame bound $B$ if $N\neq M$, and upper frame bound $B-\lambda_M$ if $N=M$. 
\end{theorem}

\begin{proof}
Letting $G$ be the synthesis operator of $\{g_n\}_{n=1}^N$, since the eigenvalues of the frame operator $GG^*$ are precisely the non-zero eigenvalues of the Gram matrix $G^*G$, this is immediate from equation (\ref{bnd}).
\end{proof}

The previous theorem shows a Naimark complement of a frame may not have a lower frame bound comparable to the lower frame bound of the original frame.  However, we can control the lower frame bounds by considering non-canonical completions of $\{f_n\}_{n=1}^N$ to a tight frame.  Notice a different completion will effect the isometry $\Phi^*$ given in the definition of a Naimark complement.  Specifically instead of adding the $M-K$ vectors $\{h_m\}_{m=K+1}^M=\{(B-\lambda_m)\varphi_m\}_{m=K+1}^M$ to create a $B-$tight frame, we may add the $M$ vectors $\{h_m\}_{m=1}^M=\{(C-\lambda_m)\varphi_m\}_{m=1}^M$ to produce a tight frame with any desired tight frame bound $C>B$.  In this case, $\Phi^*$ is now an isometry embedding $\cH_M$ into $\ell_2(M+N)$.  Identifying $\cH_M$ with $\Phi^*(\cH_M)$ and $\cH_{N}$ with $[\Phi^*(\cH_M)]^\perp$, arguments like those in Theorem \ref{prop} and Theorem \ref{bds} show a Naimark complement's frame vectors now span $\cH_{M}^\perp=\cH_N$, and this frame has upper frame bound $C$ and lower frame bound $C-B$.  

We also note completions with any other numbers of vectors cannot be considered if we desire a property comparable to (\ref{two}) from Theorem \ref{prop} to hold.  To be clear, suppose $\{f_n\}_{n=1}^N$ is a $B-$Bessel sequence and $\{f_n\}_{n=1}^N\cup \{h_m\}_{m=1}^R$ is any completion to a $C-$tight frame for $\cH_M$ with $C\geq B$.  Then the analysis operator for the associated Parseval frame is an isometry $\Phi^*$ embedding $\cH_M$ into $\ell_2(N+R)$.  We may define a Naimark complement $\{g_n\}_{n=1}^N$ in terms of this completion, but for (\ref{two}) from Theorem \ref{prop} to hold, we now require $\spann( \{g_n\}_{n=1}^N)=\cH_M^\perp=\cH_{N+R-M}$.  Letting $G$ be the synthesis matrix for $\{g_n\}_{n=1}^N$, we may calculate $G^*G$ similarly as in (\ref{bnd}).  It then follows that
\begin{equation*}
	N+R-M=\mbox{rank}(G)= \mbox{rank}(G^*G)= \begin{cases} N	&\mbox{if } C>B\\
																N-K	&\mbox{if } C=B.
									\end{cases}
\end{equation*}
Thus if $C>B$, we must use $R=M$ vectors for the completion of $\{f_n\}_{n=1}^N$, and if $C=B$, we must use $R=M-K$ vectors in the completion.  For the remainder of this paper, we consider the Naimark complement as originally defined where the completions to $B-$tight frames are canonical.


\section{Properties of the Naimark Complement}\label{frame}

In this section we will assume the given $B-$Bessel sequence has specific properties and show how these properties carry over to Naimark complements.

\begin{proposition} \label{properties}
Given a $B-$Bessel sequence $\{f_n\}_{n=1}^N$ for $\cH_M$ with  Naimark complement $\{g_n\}_{n=1}^N$, the following properties hold.
\begin{enumerate}[(a)]
	\item \label{one2} $\inner{g_n}{g_{n'}} = -\inner{f_n}{f_{n'}}$ for all $1\leq n \neq n' \leq N$. In particular, if $\{f_n\}_{n=1}^N$ is an equiangular frame, so is
	$\{g_n\}_{n=1}^N$.
	\item \label{three2} If $J\subset \{1,2,\ldots,N\}$ and $\{f_n\}_{n\in J}$ is an orthogonal set, then so is $\{g_n\}_{n\in J}$.
	\item \label{two2} If $J\subset \{1,2,\ldots,N\}$ and $\{f_n\}_{n\in J}$ is an equal-norm set, then so is $\{g_n\}_{n\in J}$.
\end{enumerate} 
\end{proposition}

\begin{proof}
For (\ref{one2}), since $\{f_n\oplus g_n\}_{n=1}^N$ is an orthogonal set, for all $1 \leq n \neq n' \leq N$ we have
\[ 
	0 = \inner{f_n \oplus g_n}{f_{n'} \oplus g_{n'}} = \inner{f_n}{f_{n'}} + \inner{g_n}{g_{n'}},
\]
giving the result.

Now (\ref{three2}) is immediate from (\ref{one2}).

To prove (\ref{two2}), set $\norm{f_n}=c$, for all $n\in J$.  Then
\[
	1=\frac{1}{B}(\norm{f_n}^2 + \norm{g_n}^2)
\]
so that $\norm{g_n}^2 = B-c^2$ for $n\in J$.
\end{proof}

The restricted isometry property was introduced by Candes and Tao in their pivotal paper \cite{CaT} and is an important property often leveraged in compressed sensing.

\begin{definition}
A family of unit norm vectors $\{f_n\}_{n=1}^N$ in $\cH_M$ has the $(L,\delta)$-{\it restricted isometry property} (RIP) if for every $ J \subseteq \{1,\ldots,N\}$, $\abs{J} \leq  L$ and all scalars $\{a_n\}_{n\in J}$ 
we have
\[ 
	(1-\delta)\sum_{n\in J}\abs{a_n}^2 \leq\norm{\sum_{n\in J}a_nf_n}^2\leq(1+\delta)\sum_{n\in J} \abs{a_n}^2.
\]
\end{definition}

Given a frame with the restricted isometry property, we may calculate RIP bounds for a Naimark complement in terms of the upper frame bound and RIP bounds of the original frame.

\begin{theorem}
Let $\{f_n\}_{n=1}^N$ be a  frame with upper frame bound $B$ and Naimark complement $\{g_n\}_{n=1}^N$.  If $\{f_n\}_{n=1}^N$ has the $(L,\delta)-$RIP, then $\{\frac{1}{\sqrt{B-1}}g_n\}_{n=1}^N$ has the $(L,\frac{\delta}{B-1})-$RIP.
\end{theorem}

\begin{proof}
Let $\{g_n\}_{n=1}^N$ be a Naimark complement of $\{f_n\}_{n=1}^N$.  Now, $\{\frac{1}{\sqrt{B-1}}g_n\}_{n=1}^N$ is unit norm since $B = \norm{f_n\oplus g_n}^2 = \norm{f_n}^2 +\norm{g_n}^2 = 1+\norm{g_n}^2$. For any $J\subset \{1,2,\ldots,N\}$ and any scalars $\{a_n\}_{n\in J}$, we have
\[ 
	B\sum_{n\in J} \abs{a_n}^2 = \norm{\sum_{n\in J}a_nf_n}^2 + \norm{\sum_{n\in J} a_ng_n}^2.
\]
Hence,
\begin{align*}
	\norm{\sum_{n\in J} a_ng_n}^2 	&= B\sum_{n\in J}\abs{a_n}^2 - \norm{\sum_{n\in J}a_nf_n}^2	\\
																&\geq B\sum_{n\in J}\abs{a_n}^2 - (1+\delta)\sum_{n\in J}\abs{a_n}^2	\\
																&= ((B-1)-\delta)\sum_{n\in J}\abs{a_n}^2.
\end{align*}
Dividing through this inequality by $B-1$ yields
\[
	\norm{\sum_{n\in J} a_n\frac{g_n}{\sqrt{B-1}}}^2\geq \left ( 1- \frac{\delta}{B-1}\right ) \sum_{n\in J}\abs{a_n}^2.
\]
Similarly,
\[ 
	\norm{\sum_{n\in J}a_n\frac{g_n}{\sqrt{B-1}}}^2 \leq \left ( 1 + \frac{\delta}{\sqrt{B-1}}\right ) \sum_{n\in J}\abs{a_n}^2
\]
completing the proof.
\end{proof}


\section{Fusion Frames}

We next consider general Naimark complements in terms of fusions frames.  As a frame operator may be viewed as a sum of weighted orthogonal projections onto the one-dimensional spans of the frame vectors, fusion frames generalize this to weighted projections onto subspaces of arbitrary dimensions.  Specifically, let $\{W_\ell\}_{\ell=1}^L$ be a family of subspaces in the Hilbert space $\cH_M$, and let $\nu_\ell>0$, $\ell=1,2,\ldots,L$ be positive weights.  Then $\{W_\ell,\nu_\ell\}_{\ell=1}^L$ is a {\em fusion frame} for $\cH_M$ if there are constants $0<A\leq B < \infty$ so that
\[ 
	A \|f\|^2 \le \sum_{\ell=1}^L \nu_\ell^2\|P_\ell f\|^2 \le B \|f\| \mbox{ for all } f\in \cH_M
\]
where $P_\ell$ is the orthogonal projection onto $W_\ell$.  We call $A,B$ the {\em fusion frame bounds}, and if $A=B=1$, this is a {\em Parseval fusion frame}.  The {\em fusion frame operator} $S:\cH_M\rightarrow \cH_M$ is then given by
\[
	Sf=\sum_{\ell=1}^L \nu_\ell^2 P_\ell f.
\]

Notice if we let $\{f_{\ell j}\}_{j=1}^{D_\ell}$ be an orthonormal basis for $W_\ell$ and consider the frame $F=\{\nu_\ell f_{\ell d}\}_{\ell=1,d=1}^{L,\ D_\ell}$, we have
\[
	Sf=\sum_{\ell=1}^L \nu_\ell^2 P_\ell f = \sum_{\ell=1}^L \sum_{d=1}^{D_\ell} \nu_\ell^2 \inner{f}{f_{\ell d}}f_{\ell d}= \sum_{\ell=1}^L \sum_{d=1}^{D_\ell} \inner{f}{{\nu_\ell}f_{\ell d}}{\nu_\ell}f_{\ell d}=FF^*f.
\]
Thus every fusion frame arises from a traditional frame with extra orthogonality conditions among the frame vectors.  This leads to the definition of a Naimark fusion frame for a Parseval fusion frames in a natural way.  We omit the standard definition of a Naimark fusion frame for Parseval frames and instead imediately discuss the generalized version.

\begin{definition} \label{fusdef}
Let $\{W_\ell,\nu_\ell\}_{\ell=1}^L$ be a fusion frame for $\cH_M$ with frame bounds $A,B$.  Choose orthonormal bases $\{f_{\ell d}\}_{d=1}^{D_\ell}$ for $W_\ell$, and consider the frame $\{\nu_\ell f_{\ell d}\}_{\ell=1,d=1}^{L,\ D_\ell}$ for $\cH_M$.  Construct the general Naimark complement $\{g_{\ell d}\}_{\ell=1,d=1}^{L,\ D_\ell}$ as in Definition \ref{naimarkdef}.  Then for each $\ell=1,2,\ldots,L$, $\{g_{\ell d}\}_{d=1}^{D_\ell}$ is an orthogonal set where each vector has norm $\sqrt{B-\nu_\ell^2}$.  Set
\[
	W'_\ell = \spann(\{g_{\ell d}\}_{d=1}^{D_\ell}).
\]
A Naimark fusion frame is then given by $\{W'_\ell,\sqrt{B-\nu_\ell^2}\}_{\ell=1}^L$. 
\end{definition}

\begin{remark}\label{rmk2}
As with Naimark complements discussed in Remark \ref{rmk}, this Naimark fusion frame is formally embedded in $\ell_2(M+N-K)$ where $N=\sum_{\ell=1}^L D_\ell$.  For fusion frames however, the embedding is not unique.  Indeed, the analysis operator $\Phi^*: \cH_{M}\rightarrow \ell_2(M+N-K)$ for the Parseval frame $\{\frac{1}{\sqrt{B}}f_{\ell d}\}_{\ell=1,d=1}^{L,\ D_\ell} \cup \{\frac{1}{\sqrt{B}}h_m\}_{m=K+1}^{M}$ now depends on the choice of orthonormal basis $\{f_{\ell d}\}_{d=1}^{D_\ell}$ for each subspace $W_\ell$.  However, once we fix an orthonormal bases for each $W_\ell$, $\Phi^*$ is determined, and we may identify $\cH_M$ with $\Phi^*(\cH_M)$ so that $\{W_\ell,\nu_\ell\}_{\ell=1}^L$ is a fusion frame for $\cH_M\subseteq \ell_2(M+N-K)$ and $\{W'_\ell,\sqrt{B-\nu_\ell^2}\}_{\ell=1}^L$ is a Naimark fusion frame for $\cH_{N-K}\subseteq \ell_2(M+N-K)$ where $\cH_{M} \oplus \cH_{N-K}=\ell_2(M+N-K)$.  This again allows us to approach Naimark fusion frames from the perspective of matrix completion.  Further, no matter what orthonormal bases are chosen for the $W_\ell$, the Naimark fusion frames are unitarily equivalent as we show in the next theorem.  The proof is a version of an unpublished argument of Jameson Cahill and Dustin Mixon which we use with their permission.
\end{remark}

\begin{theorem} \label{unitequiv}
Let $\{W_\ell,\nu_\ell\}_{\ell=1}^L$ be a fusion frame for $\cH_M$ with frame bounds $A,B$ and with Naimark fusion frames $\{W'_\ell,\sqrt{B-\nu_\ell^2}\}_{\ell=1}^L$ and $\{Z'_\ell,\sqrt{B-\nu_\ell^2}\}_{\ell=1}^L$.  Then there exists a unitary operator $U$ such that $UZ'_\ell=W'_\ell$ for all $\ell=1,\cdots,L$.
\end{theorem}

\begin{proof}
For each $W_\ell$, fix an orthonormal basis $\{f_{\ell d}\}_{d=1}^{D_\ell}$ and let $F_\ell$ by the $M \times D_\ell$ synthesis matrix for $\{\nu_\ell f_{\ell d}\}_{d=1}^{D_\ell}$.  Let $\{g_{\ell d}\}_{\ell =1, d=1}^{L, \ D_\ell}$ be a Naimark complement of $\{\nu_\ell f_{\ell d}\}_{\ell=1, d=1}^{L, \ D_\ell}$ as in Definition \ref{fusdef} so that $\spann(\{g_{\ell d}\}_{d=1}^{D_\ell})=W'_\ell$.  Let $G_\ell$ be the $(N-K)\times D_\ell$ synthesis matrix for $\{g_{\ell d}\}_{d=1}^{D_\ell}$.  Then in terms of matrix completion, there exists $H_2$ such that
\[
	\frac{1}{\sqrt{B}}\begin{bmatrix}
		F_1	&	F_2	&\cdots	&F_K	&H_1	\\
		G_1	&	G_2	&\cdots	&G_K	&H_2	
	\end{bmatrix}
\]
is a unitary matrix.  Now suppose we choose a different orthonormal basis for each $W_\ell$.  That is, let $U_\ell$ be a $D_\ell \times D_\ell$ unitary matrix, and take the (scaled) orthonormal basis of $W_\ell$ to be the column vectors of the synthesis matrix $F_\ell U_\ell$ for each $\ell=1,\ldots,L$.  If we also consider $G_\ell U_\ell$, notice
\[
	\frac{1}{\sqrt{B}}\begin{bmatrix}
		F_1U_1		&\cdots	&F_LU_L	&H_1	\\
		G_1U_1		&\cdots	&G_LU_L	&H_2	
	\end{bmatrix}
	=		\frac{1}{\sqrt{B}}\begin{bmatrix}
		F_1	&	F_2	&\cdots	&F_K	&H_1	\\
		G_1	&	G_2	&\cdots	&G_K	&H_2	
	\end{bmatrix}
	\begin{bmatrix}
		U_1	&			&				&			&\\
				&U_2	&				&			&\\
				&			&\ddots	&			&\\
				&			&				&U_L	&\\
				&			&				&			&I\\
	\end{bmatrix}.
\]
This is a unitary matrix since the right hand side of this equality is a product of two unitary matrices.  Thus a frame with synthesis matrix $\begin{bmatrix} F_1U_1 & \cdots &F_LU_L \end{bmatrix}$ has a Naimark complement with synthesis matrix $\begin{bmatrix} G_1U_1 & \cdots &G_LU_L \end{bmatrix}$.  Further, as the range of each synthesis matrix is the span of the frame vectors, we have $\mbox{range}(G_\ell U_\ell)=\mbox{range}(G_\ell)=W'_\ell$ for all $\ell=1,\ldots,L$. 

Now let $\{Z'_\ell,\sqrt{B-\nu_\ell^2}\}_{\ell=1}^L$ be any Naimark fusion frame of $\{W_\ell,\sqrt{B-\nu_\ell^2}\}_{\ell=1}^L$.  The subspaces $Z'_\ell$ arise from $Z'_\ell=\spann(\{z_{\ell d}\}_{d=1}^{D_\ell})$ where $\{z_{\ell d}\}_{\ell=1,d=1}^{L, \ D_\ell}$ is a Naimark complement of a frame with a synthesis matrix of the form $\begin{bmatrix} F_1U_1 & \cdots &F_LU_L \end{bmatrix}$.  By Theorem \ref{prop}(\ref{three}), there exists some unitary operator $U$ such that $UZ'_\ell=\mbox{range}(G_\ell U_\ell)$ for all $\ell=1,\ldots, L$.  Thus
\[
	UZ'_\ell=\mbox{range}(G_\ell U_\ell)=\mbox{range}(G_\ell)=W'_\ell
\]
for all $\ell=1,\ldots,L$ completing the proof.
\end{proof}

There are many ways to measure the distance between two subspaces of a Hilbert space.  The most exact measure comes from the {\em principal angles}.  Intuitively we find two unit norm vectors (one in each subspace) with the minimal angle formed between them. Then we consider the orthogonal complements of these vectors in their respective subspaces, and find the closest two unit norm vectors in these subspaces.  We continue in this manner until one of the orthogonal complements is zero.  The formal definition follows. For notation, if $W$ is a subspace of $\cH_M$, we write $S_{W} = \{f\in W:\norm{f}=1\}$. 

\begin{definition}
Given two subspaces $W_1,W_2$ of $\cH_M$ with dim $W_1=d_1\le$ dim $W_2=d_2$, the {\em principal angles} $(\theta_1,\theta_2,\ldots \theta_{d_1})$ between the subspaces are defined as
\[ 
	\theta_1 =\min\{\arccos \inner{f}{g}: f\in S_{W_1},g\in S_{W_2}\}.
\]
Two vectors $f_1,g_1$ are called {\em principal vectors} if they give the minimum above. The remaining principal angles and vectors are defined recursively via
\[ 
	\theta_j = \min\{\arccos\ \langle f,g\rangle: f\in S_{W_1},g\in S_{W_2},  \mbox{ and } f\perp f_\ell,g\perp g_\ell, 1\le \ell \le j-1.\}
\]
\end{definition}

Now we will consider how principal angles are passed from a fusion frame to the Naimark fusion frame.  In definition \ref{fusdef}, for $\ell=1,\ldots,L$, we choose $\{f_{\ell d}\}_{d=1}^{D_\ell}$ as an orthonormal basis for each subspace $W_\ell$.  By (\ref{three2}) of Theorem \ref{properties}, the Naimark fusion frame subspace $W'_\ell$ must also have dimension $D_\ell$.  Thus the number of principal angles and vectors between $W_n$ and $W_m$ compared to those between $W'_n$ and $W'_m$, $1\leq n,m\leq L$ must be the same.  To calculate these principal angles and vectors for the Naimark fusion frame, we will use the following theorem which we prove for Parseval fusion frames.

\begin{theorem}\label{L1}
Let $\{W_\ell,\nu_\ell\}_{\ell=1}^L$ be a Parseval fusion frame for $\cH_M$ with $\dim W_\ell=D$ for all $\ell=1,2,\ldots,L$, and let $\{W'_\ell,\sqrt{1-\nu_\ell^2}\}_{\ell=1}^L$ be a Naimark complement.  Fix two subspaces, say $W_{1}, W_{2}$, and suppose $\{\theta_d\}_{d=1}^D$ are the associated principal angles.  Let $N=DL$, and assume we have an embedding of the fusion frame into $\ell_2(N)$ (see remark \ref{rmk2}) with $P$ the orthogonal projection of $\ell_2(N)$ onto $\cH_M$. Let $\{e_{\ell d}\}_{\ell=1,d=1}^{L\ D}$ be an orthonormal basis for $\ell_2(N)$ which satisfies
\begin{enumerate}[(a)]
	\item $\{\frac{1}{\nu_\ell}Pe_{\ell d}\}_{d=1}^D$, $\ell=1,2$ are the principal vectors for $W_1,W_2$.
	\item $W_\ell =\spann \{Pe_{\ell d}\}_{d=1}^D$, for all $\ell=1,2,\ldots,L$.
\end{enumerate}
Then the complementary fusion frame subspaces $W'_{1}, W'_{2}$ have principal vectors $\{\frac{1}{\sqrt{1-\nu_\ell^2}}(I-P)e_{\ell d}\}_{d=1}^D$, $\ell=1,2$ and principal angles
\[
	\{\theta_d'\}_{d=1}^D=\left \{ \arccos\left [ \frac{\nu_1}{\sqrt{1-\nu_1^2}}\frac{\nu_2}{\sqrt{1-\nu_2^2}} \cos(\theta_d) \right ]\right \}_{d=1}^D.
\]
\end{theorem}

\begin{proof}
We will find the first principal vectors and associated principal angle between $W'_1$ and $W'_2$; the result will follow by iteration of the argument.
To identify the first principal vectors we need to maximize
\[ 
	\{\inner{f}{g}: f\in S_{W'_1},g\in S_{W'_2}\}.
\]
That is, we need to maximize
\begin{equation}\label{E1} 
	 \left \langle \sum_{d=1}^D \frac{a_d}{\sqrt{1-\nu_1^2}}(I-P)e_{1d}, \sum_{d'=1}^D\frac{b_{d'}}{\sqrt{1-\nu_2^2}}(I-P)e_{2d'}\right \rangle,
\end{equation}
subject to the constraints $\sum_{d=1}^D \abs{a_d}^2 = \sum_{d'=1}^D \abs{b_{d'}}^2 =1$.  From equation (\ref{E1}), we need to maximize
\begin{align}
 	\frac{1}{\sqrt{1-\nu_1^2}}	&\frac{1}{\sqrt{1-\nu_2^2}} \sum_{d=1}^D\sum_{d'=1}^D a_d\overline{b_{d'}}\langle (I-P)e_{1d},(I-P)e_{2d'}\rangle \notag\\
															&=- \frac{1}{\sqrt{1-\nu_1^2}}\frac{1}{\sqrt{1-\nu_2^2}} \sum_{d=1}^D\sum_{d'=1}^D a_d\overline{b_{d'}}\langle Pe_{1d},Pe_{2d'}\rangle \label{E2}\\
															&=- \frac{\nu_1}{\sqrt{1-\nu_1^2}}\frac{\nu_2}{\sqrt{1-\nu_2^2}} \left \langle \sum_{d=1}^D \frac{a_d}{\nu_1}Pe_{1d},\sum_{d'=1}^D \frac{b_{d'}}{\nu_2}Pe_{2d'} \right \rangle. \notag
\end{align}
However, since $\{\frac{1}{\nu_\ell}Pe_{\ell d}\}_{d=1}^D$, $\ell=1,2$, are the principal vectors for $W_1,W_2$, the maximum in equation (\ref{E2}) is precisely
\[ 
	\frac{\nu_1}{\sqrt{1-\nu_1^2}}\frac{\nu_2}{\sqrt{1-\nu_2^2}}\cos( \theta_d).
\]
Observing that the inner product
\[ 
	\left \langle \frac{1}{\sqrt{1-\nu_1^2}}(I-P)e_{11},\frac{1}{\sqrt{1-\nu_2^2}}(I-P)e_{21} \right \rangle.
\]
yields precisely this value, these are the principal vectors associated with $\theta'_1$.
\end{proof}

Note it is not necessary to have subspaces of equal dimension since the same proof holds.  Now the result easily generalizes to any fusion frame $\{W_\ell,\nu_\ell\}_{\ell=1}^L$ with fusion frame bounds $A,B$.  In this case, after we pick orthonormal bases for each subspace, weight them by the appropriate $\nu_\ell$, complete to a $B$-tight frame, and normalize  by $1/\sqrt{B}$ to obtain a Parseval frame Theorem \ref{L1} applies where we simply replace $\nu_\ell$ by $\nu_\ell/\sqrt{B}$.

\begin{corollary}\label{cor1}
Let $\{W_\ell,\nu_\ell\}_{\ell=1}^L$ be a fusion frame for $\cH_M$ with $\dim W_\ell=D_\ell$ for each $\ell=1,2,\ldots,L$ and upper frame bound $B$.  Let $\{W'_\ell,\sqrt{B-\nu_\ell^2}\}_{\ell=1}^L$ be a Naimark fusion frame.  Fix subspaces, say $W_1,W_{2}$ where $D_1\leq D_{2}$, and suppose $\{\theta_{d}\}_{d=1}^{D_1}$ are the associated principal angles.  Then the principal angles for the subspaces $W'_1,W'_{2}$ are
\[ 
	\{\theta'_d\}_{d=1}^{D_1}=\left \{\arccos \left [ \frac{\nu_1}{\sqrt{B-\nu_1^2}}\frac{\nu_{2}}{\sqrt{B-\nu_{2}^2}} \cos(\theta_{d})\right ] \right \}_{d=1}^{D_1}.
\]
\end{corollary}

Another measure of distance between subspaces of a Hilbert space is the {\em chordal distance} so named as it can be expressed as a multiple of the straight line distance between projection matrices living on a sphere.  This distance is closely related to fusion frames with maximal resilience to noise and erasures \cite{KPC}.  While there are several equivalent forms for this distance, as we have already calculated principal angles for Naimark fusion frames, we will use the following from \cite{CHS}.

\begin{definition}
If $W_1,W_2$ are subspaces of $\cH_M$ of dimension $D$, the {\em chordal distance} $d_c(W_1,W_2)$ between the subspaces is given by
\[ 
d_c^2(W_1,W_2) = D - Tr[P_1P_2] = D-\sum_{d=1}^D \cos^2\theta_{d},
\]
where $P_\ell$ is the orthogonal projection onto $W_\ell$, $\ell=1,2$, and $\{\theta_{d}\}_{d=1}^D$ are the principal angles between $W_1,W_2$.
\end{definition}

Using Corollary \ref{cor1}, we can correct Theorem 3.6 of \cite{CCHKP} which computes the chordal distances for the Naimark complement of a Parseval fusion frame incorrectly.

\begin{theorem}
Let $\{W_\ell,\nu_\ell\}_{\ell=1}^L$ be a fusion frame for $\cH_M$ with upper frame bound B, and let $\{W'_\ell,\sqrt{1-\nu_\ell^2}\}_{\ell=1}^L$ be a Naimark fusion frame.  Fix two subspaces, say $W_1,W_2$, each of dimension $D$.  Then
\[ 
	d_c^2(W'_1,W'_{2}) = \left [ 1 - \frac{\nu_1^2}{1-\nu_1^2}\frac{\nu_{2}^2}{1-\nu_{2}^2} \right ]D + \left [ \frac{\nu_1^2}{1-\nu_1^2}\frac{\nu_{2}^2}{1-\nu_{2}^2} \right ] d_c^2(W_1,W_{2}).
\]
\end{theorem}

\begin{proof}
Using Corollary \ref{cor1} we perform the calculation
\begin{align*}
	d_c^2(W'_1,W'_{2}) 
		&= D - \sum_{d=1}^D\frac{\nu_1^2}{B-\nu_1^2}\frac{\nu_{2}^2}{B-\nu_{2}^2}\cos^2\theta_{d}\\
		&= D-\left [ \frac{\nu_1^2}{B-\nu_1^2}\frac{\nu_{2}^2}{B-\nu_{2}^2} \right ] \sum_{d=1}^D \cos^2\theta_{d}\\
		&= D-\left [ \frac{\nu_1^2}{B-\nu_1^2}\frac{\nu_{2}^2}{B-\nu_{2}^2} \right ] \left [ D- d_c^2(W_1,W_2)\right ]\\
		&= \left [ 1 - \frac{\nu_1^2}{B-\nu_1^2}\frac{\nu_{2}^2}{B-\nu_{2}^2} \right ]D + \left [ \frac{\nu_1^2}{B-\nu_1^2}\frac{\nu_{2}^2}{B-\nu_{2}^2}\right ] d_c^2(W_1,W_{2}).
\end{align*}
\end{proof}

\section*{Acknowledgments}
Casazza/Smalyanau were supported by NSF DMS 1008183, NSF ATD 1042701, and AFOSR FA9550-11-1-0245.
Fickus was supported by NSF DMS 1042701, NSF CCF 1017278, AFOSR F1ATA00083G004 and AFOSR F1ATA00183G003. 
Mixon was supported by the A.B. Krongard Fellowship.

\end{document}